\documentclass[reqno, 12pt]{amsart}

\usepackage{amsmath,amssymb,amsthm,amsfonts,mathrsfs}
\usepackage{fullpage}
\usepackage{xcolor}
\usepackage{graphicx}
\usepackage{makecell}
\usepackage{listings}
\usepackage{bbm}
\usepackage{booktabs}
\usepackage{pgfplots}

\definecolor{codegreen}{rgb}{0,0.6,0}
\definecolor{codegray}{rgb}{0.5,0.5,0.5}
\definecolor{codepurple}{rgb}{0.58,0,0.82}
\definecolor{backcolour}{rgb}{0.95,0.95,0.92}

\lstdefinestyle{mystyle}{
  backgroundcolor=\color{backcolour},   commentstyle=\color{codegreen},
  keywordstyle=\color{magenta},
  numberstyle=\tiny\color{codegray},
  stringstyle=\color{codepurple},
  basicstyle=\ttfamily\footnotesize,
  breakatwhitespace=false,         
  breaklines=true,                 
  captionpos=b,                    
  keepspaces=true,                 
  numbers=left,                    
  numbersep=5pt,                  
  showspaces=false,                
  showstringspaces=false,
  showtabs=false,                  
  tabsize=2
}

\usepackage{listings}
\usepackage[colorlinks=true,
linkcolor=violet,
anchorcolor=blue,
citecolor=violet
]{hyperref}
\allowdisplaybreaks 
\newtheorem{theorem1number}{Theorem}
\newtheorem{theorem}{Theorem}[section]
\newtheorem{lemma}[theorem]{Lemma}

\newtheorem*{conjecture*}{Conjecture}
\newtheorem*{theorem*}{Theorem}
\theoremstyle{definition}

\theoremstyle{remark}

\newtheorem*{remark*}{Remark}

\usepackage{colonequals}

\author{\textbf{Runbo Li}}
\address{Department of Mathematics, Imperial College London, London SW7 2AZ, UK}
\email{runbo.li26@imperial.ac.uk, runbo.li.carey@gmail.com}

\makeatletter
\@namedef{subjclassname@2020}{\textup{}2020 Mathematics Subject Classification}
\makeatother

\title[]{Large value estimates for Dirichlet polynomials and trilinear sums over short intervals}
\subjclass[2020]{\textbf{11M06}, \textbf{11N05}, \textbf{11N35}, \textbf{11N36}}
\keywords{\textbf{Large value estimates}, \textbf{Distribution of primes}, \textbf{Short intervals}}

\begin{document}

\begin{abstract}
Recently, Matom\"{a}ki, Pandey, Pliego, Ter\"{a}v\"{a}inen and Wang \cite{MPPTW} improved the Baker--Harman--Pintz \cite{BHP} estimate of trilinear sums over short intervals of length $x^{\theta}$, which naturally reduces to estimating mean values of products of three Dirichlet polynomials. Based on recent works of Guth and Maynard \cite{GuthMaynard} and Chen, Gupta and Li \cite{ChenGuptaLi} on new large value estimates for Dirichlet polynomials, they improved the Baker--Harman--Pintz estimate for $\theta > \frac{3}{5}$. In this paper, we use their method to improve the Baker--Harman--Pintz estimate for every $\frac{1}{2} < \theta < 1$. Our result yields improvements on many problems, including the work of Li \cite{LRB052} on primes in short intervals of length $x^{0.52}$.
\end{abstract}

\maketitle

\tableofcontents

\section*{Notation}
Throughout this paper, we always assume that $x$ is a sufficiently large integer, $\varepsilon$ is a sufficiently small positive number (sometimes depends on $\theta$) and $K \geqslant 1$. The appearance of $K$ in the exponent of a logarithm will always signify that the result holds for every $K \geqslant 1$ with an implied constant that depends on $K$. We shall ignore the presence of the integer multiples of $\varepsilon$ in many places below. The letter $p$, with or without subscript, is reserved for prime numbers. Let $c$ and $c_i$ ($i \in \mathbb{N}$) denote positive constants which may have different values at different places. We use both Landau and Vinogradov notations, and implied constants may depend on fixed parameters. We write $m \sim M$ to mean that $M \leqslant m < 2M$ and $m \asymp M$ to mean that $c_1 M < m < c_2 M$. Let $\frac{1}{2} < \theta < 1$, $y = x^{\theta + \varepsilon}$, $y_1 = x \exp \left(- 3 (\log x)^{1/3} \right)$, $T = x^{1-\theta}$ and $T_0 = \exp \left((\log x)^{1/3} \right)$. We use $M(s)$, $N(s)$, $R(s)$ and some other capital letters, with or without subscript, to denote some divisor-bounded Dirichlet polynomials
$$
M(s)=\sum_{m \sim M} a_{m} m^{-s}, \quad N(s)=\sum_{n \sim N} b_{n} n^{-s}, \quad R(s)=\sum_{r \sim R} c_{r} r^{-s}.
$$
Note that for a divisor-bounded Dirichlet polynomial $M(s)$, there exists $c_d$ such that for every $d \in \mathbb{N}$,
$$
\sum_{m \leqslant M}|a_m|^{2d} \ll M(\log M)^{c_d}.
$$
We say a Dirichlet polynomial $M(s)$ is \textit{prime-factored} if we have
$$
\left|M \left( \frac{1}{2}+it \right)\right| \ll M^{\frac{1}{2}} (\log x)^{-K}
$$
for $T_0 < |t| < T$. In fact, this holds when $a_m$ is the characteristic function of primes or of numbers with a bounded number of prime factors restricted to certain ranges. For example, if we have
$$
a_m = \sum_{m = p_1 \cdots p_k} 1
$$
and the least prime factor of $m$ is $\gg \exp \left((\log x)^{4/5} \right)$, then we know that $M(s)$ is prime-factored by [\cite{HarmanBOOK}, Lemma 1.5].

\section{Introduction}
In this paper we seek the mean value estimates of products of three Dirichlet polynomials in the following form:
\begin{equation}\label{eq:form1}
\int_{T_0}^{T} \left|M \left( \frac{1}{2}+it \right) N \left( \frac{1}{2}+it \right) R \left( \frac{1}{2}+it \right)\right| \ll x^{\frac{1}{2}} (\log x)^{-K},
\end{equation}
where $M, N, R \gg x^{\varepsilon}$ and $M N R \asymp x$. This type of mean value estimates plays a crucial role in finding primes in short intervals. In fact, by [\cite{BHP}, Lemma 6] or [\cite{HarmanBOOK}, Lemma 7.2], \eqref{eq:form1} implies the following asymptotic formula for trilinear sums:
\begin{equation}\label{eq:asymp}
\sum_{\substack{m n r \in [x - y, x) \\ m \sim M \\ n \sim N}} a_m b_n c_r = \frac{y}{y_1} \sum_{\substack{m n r \in [x - y_1, x) \\ m \sim M \\ n \sim N}} a_m b_n c_r + O\left(y(\log x)^{-K}\right).
\end{equation}
By various asymptotic formulas of the form \eqref{eq:asymp}, together with a combinatorial decomposition of prime counting function (usually done by applying Heath-Brown identity or Harman's sieve), one would get an asymptotic formula or a nontrivial lower bound for the number of primes in a short interval $[x - y, x)$.

In 1996, Baker, Harman and Pintz \cite{BHP2} proved \eqref{eq:form1} under the condition that $R(s)$ is prime-factored and $M, N$ lie in specific regions. Write $M = x^{\alpha_1}$ and $N = x^{\alpha_2}$, then their restrictions on $M, N$ becomes
\begin{equation}\label{eq:condition_gamma}
|\alpha_1 - \alpha_2| < 2 \theta - 1, \qquad 1 - \alpha_1 - \alpha_2 < \gamma(\theta),
\end{equation}
where
$$
\gamma(\theta) = \min \left(4 \theta - 2,\ \frac{20 \theta - 9}{11},\ \frac{72 \theta - 37}{11}\right).
$$
Later in 2001, they \cite{BHP} further generalize this result by replacing $\gamma(\theta)$ with a more flexible expression:
$$
\gamma(\theta) = \max_{g \in \mathbb{N}} \gamma_{g}(\theta), \quad \gamma_{g}(\theta) = \min \left(4 \theta - 2,\ \frac{(8 g - 4) \theta - (4 g - 3)}{4 g - 1},\ \frac{24 g \theta - (12 g + 1)}{4 g - 1}\right).
$$
Taking $g = 3$, this reduces to their 1996 result. Using this 2001 result (with $g = 5$) as the main Type-II arithmetic information input to give asymptotic formulas of the form \eqref{eq:asymp}, together with an application of Harman's sieve, they showed that there is always a prime number in the short interval $[x - y, x]$, where $\theta = 0.525$. In 2025, Li \cite{LRB052} improved this to $\theta = 0.52$, and he also used this 2001 result (with $g = 6$) to produce asymptotic formulas of the form \eqref{eq:asymp}.

The proof of the 2001 Baker--Harman--Pintz result uses large value estimates for Dirichlet polynomials, but only classical bounds from classical mean value estimate and Halász--Montgomery--Huxley large value estimate were available at that time. In 2026, Guth and Maynard \cite{GuthMaynard} and Chen, Gupta and Li \cite{ChenGuptaLi} proved new large value estimates for Dirichlet polynomials, which is better than classical bounds in some cases. Recently, Matom\"{a}ki, Pandey, Pliego, Ter\"{a}v\"{a}inen and Wang \cite{MPPTW} improved the 2001 Baker--Harman--Pintz result for $\theta > \frac{3}{5}$ by developing a new method and inserting those new large value estimates into the proof. In this paper, by applying the method developed in \cite{MPPTW}, we obtain an improvement over the 2001 Baker--Harman--Pintz result for any $\frac{1}{2} < \theta < 1$.
\begin{theorem1number}\label{t1}
Let $\frac{1}{2} < \theta < 1$. Write $M = x^{\alpha_1}$ and $N = x^{\alpha_2}$. Then \eqref{eq:form1} holds if one of $M(s)$, $N(s)$, $R(s)$ is prime-factored and
$$
|\alpha_1 - \alpha_2| < 2 \theta - 1, \qquad 1 - \alpha_1 - \alpha_2 < \gamma^{*}(\theta) - \varepsilon,
$$
where
\begin{align*}
\gamma^{*}(\theta) =
\begin{cases}
\max\left(\Gamma_{0}(\theta),\ u_{k_0(\theta)}(\theta) \right), & \quad \frac{1}{2} < \theta \leqslant \frac{31}{55}, \\
20 \theta - 11, & \quad \frac{31}{55} < \theta \leqslant \frac{17}{30}, \\
\frac{5 \theta - 1}{4}, & \quad \frac{17}{30} < \theta < 1, \\
\end{cases}
\end{align*}
$$
\Gamma_{0}(\theta) = \max_{k \geqslant 2} \min\left(u_k(\theta),\ p_k(\theta),\ q_k(\theta) \right), \qquad k_0(\theta) = \max\left(2,\ \left\lfloor \frac{\theta - \Gamma_{0}(\theta)}{4 \theta - 2} \right\rfloor + 1 \right),
$$
$$
u_k(\theta) = \frac{(4 k - 3) \theta - (2 k - 2)}{2 k - 1}, \quad q_k(\theta) = \frac{(4 k - 1)(2 \theta - 1)}{2 k - 1}, \qquad (k \in \mathbb{N})
$$
$$
p_k(\theta) = \frac{5((8 k - 1) \theta - 4 k)}{4 k - 5}. \qquad (k \in \mathbb{N},\ k \geqslant 2)
$$
\end{theorem1number}
\begin{remark*}
When $\theta = 0.52$, the 2001 Baker--Harman--Pintz result $\gamma_{\operatorname{BHP}}(\theta)$ gives $\frac{2}{25} = 0.08$, while our new $\gamma^{*}(\theta)$ gives $\frac{19}{225} \approx 0.0844$, giving a $5.5\%$ improvement.
\end{remark*}

\newpage
\begin{tabular}{cccc|cccc}
\hline
$\theta$ & $\gamma^{*}(\theta)$ & Exact & Attained by & $\theta$ & $\gamma^{*}(\theta)$ & Exact & Attained by \\
\hline
$0.501$ & $0.00401$ & $83/20700$ & $\Gamma_0$, $k=104$ & $0.526$ & $0.11971$ & $419/3500$ & $u_{k_0}$, $k_0=4$ \\
$0.502$ & $0.00805$ & $33/4100$ & $u_{k_0}$, $k_0=62$ & $0.527$ & $0.12157$ & $851/7000$ & $u_{k_0}$, $k_0=4$ \\
$0.503$ & $0.01214$ & $983/81000$ & $u_{k_0}$, $k_0=41$ & $0.528$ & $0.12343$ & $108/875$ & $u_{k_0}$, $k_0=4$ \\
$0.504$ & $0.01616$ & $103/6375$ & $\Gamma_0$, $k=26$ & $0.529$ & $0.12529$ & $877/7000$ & $u_{k_0}$, $k_0=4$ \\
$0.505$ & $0.02024$ & $83/4100$ & $\Gamma_0$, $k=21$ & $0.530$ & $0.13200$ & $33/250$ & $\Gamma_0$, $k=3$ \\
$0.506$ & $0.02436$ & $67/2750$ & $\Gamma_0$, $k=17$ & $0.531$ & $0.13640$ & $341/2500$ & $\Gamma_0$, $k=3$ \\
$0.507$ & $0.02848$ & $413/14500$ & $\Gamma_0$, $k=15$ & $0.532$ & $0.14080$ & $88/625$ & $\Gamma_0$, $k=3$ \\
$0.508$ & $0.03297$ & $239/7250$ & $u_{k_0}$, $k_0=15$ & $0.533$ & $0.15940$ & $797/5000$ & $u_{k_0}$, $k_0=3$ \\
$0.509$ & $0.03678$ & $423/11500$ & $\Gamma_0$, $k=12$ & $0.534$ & $0.16120$ & $403/2500$ & $u_{k_0}$, $k_0=3$ \\
$0.510$ & $0.04130$ & $19/460$ & $u_{k_0}$, $k_0=12$ & $0.535$ & $0.16300$ & $163/1000$ & $u_{k_0}$, $k_0=3$ \\
$0.511$ & $0.04529$ & $77/1700$ & $\Gamma_0$, $k=9$ & $0.536$ & $0.16480$ & $103/625$ & $u_{k_0}$, $k_0=3$ \\
$0.512$ & $0.04968$ & $118/2375$ & $u_{k_0}$, $k_0=10$ & $0.537$ & $0.16660$ & $833/5000$ & $u_{k_0}$, $k_0=3$ \\
$0.513$ & $0.05465$ & $929/17000$ & $u_{k_0}$, $k_0=9$ & $0.538$ & $0.16840$ & $421/2500$ & $u_{k_0}$, $k_0=3$ \\
$0.514$ & $0.05815$ & $189/3250$ & $\Gamma_0$, $k=7$ & $0.539$ & $0.17020$ & $851/5000$ & $\Gamma_0$, $k=3$ \\
$0.515$ & $0.06233$ & $187/3000$ & $u_{k_0}$, $k_0=8$ & $0.540$ & $0.17200$ & $43/250$ & $\Gamma_0$, $k=3$ \\
$0.516$ & $0.06646$ & $108/1625$ & $\Gamma_0$, $k=7$ & $0.541$ & $0.19133$ & $287/1500$ & $\Gamma_0$, $k=2$ \\
$0.517$ & $0.07115$ & $37/520$ & $u_{k_0}$, $k_0=7$ & $0.542$ & $0.19600$ & $49/250$ & $\Gamma_0$, $k=2$ \\
$0.518$ & $0.07527$ & $207/2750$ & $\Gamma_0$, $k=6$ & $0.543$ & $0.23833$ & $143/600$ & $u_{k_0}$, $k_0=2$ \\
$0.519$ & $0.08173$ & $899/11000$ & $u_{k_0}$, $k_0=6$ & $0.544$ & $0.24000$ & $6/25$ & $u_{k_0}$, $k_0=2$ \\
$0.520$ & $0.08444$ & $19/225$ & $\Gamma_0$, $k=5$ & $0.545$ & $0.24167$ & $29/120$ & $u_{k_0}$, $k_0=2$ \\
$0.521$ & $0.08867$ & $133/1500$ & $\Gamma_0$, $k=5$ & $0.546$ & $0.24333$ & $73/300$ & $u_{k_0}$, $k_0=2$ \\
$0.522$ & $0.09711$ & $437/4500$ & $u_{k_0}$, $k_0=5$ & $0.547$ & $0.24500$ & $49/200$ & $u_{k_0}$, $k_0=2$ \\
$0.523$ & $0.09900$ & $99/1000$ & $u_{k_0}$, $k_0=5$ & $0.548$ & $0.24667$ & $37/150$ & $u_{k_0}$, $k_0=2$ \\
$0.524$ & $0.10286$ & $18/175$ & $\Gamma_0$, $k=4$ & $0.549$ & $0.24833$ & $149/600$ & $u_{k_0}$, $k_0=2$ \\
$0.525$ & $0.10714$ & $3/28$ & $\Gamma_0$, $k=4$ & $0.550$ & $0.25000$ & $1/4$ & $u_{k_0}$, $k_0=2$ \\
\hline
\end{tabular}
\begin{center}
\textbf{Table 1.} Values of $\gamma^{*}(\theta)$
\end{center}

\begin{tikzpicture}
\begin{axis}[
  width=0.95\textwidth, height=7.5cm,
  xmin=0.5, xmax=0.55, ymin=0, ymax=0.26,
  xtick={0.50,0.51,0.52,0.53,0.54,0.55},
  xlabel={$\theta$}, ylabel={$\gamma^{*}(\theta)$},
  unbounded coords=jump,
  legend pos=north west, legend cell align=left,
  tick label style={/pgf/number format/fixed, /pgf/number format/precision=3},
  grid=major, grid style={gray!20}
]
\addplot[blue, line width=0.9pt] coordinates {
  (0.500000,0.000000) (0.500100,0.000400) (0.500200,0.000800) (0.500300,0.001201) (0.500400,0.001602)
  (0.500500,0.002002) (0.500600,0.002403) (0.500700,0.002806) (0.500800,0.003206) (0.500900,0.003608)
  (0.501000,0.004010) (0.501100,0.004417) (0.501200,0.004814) (0.501300,0.005216) (0.501400,0.005619)
  (0.501500,0.006022) (0.501600,0.006425) (0.501700,0.006837) (0.501800,0.007236) (0.501900,0.007635)
  (0.502000,0.008049) (0.502100,0.008456) (0.502200,0.008847) (0.502300,0.009252) (0.502400,0.009656)
  (0.502500,0.010060) (0.502600,0.010466) (0.502700,0.010870) (0.502800,0.011315) (0.502900,0.011682)
  (0.503000,0.012136) (0.503100,0.012493) (0.503200,0.012898) (0.503300,0.013305) (0.503400,0.013711)
  (0.503500,0.014196) (0.503600,0.014609) (0.503700,0.014935) (0.503800,0.015343) (0.503900,0.015747)
  (0.504000,0.016157) (0.504100,0.016605) (0.504200,0.016971) (0.504300,0.017383) (0.504400,0.017811)
  (0.504500,0.018200) (0.504600,0.018604) (0.504700,0.019112) (0.504800,0.019423) (0.504900,0.019904)
  (0.505000,0.020244) (0.505048,0.020438) (0.505048,nan) (0.505048,0.020627) (0.505133,0.020795)
  (0.505260,0.021311) (0.505260,nan) (0.505260,0.021515) (0.505351,0.021695) (0.505491,0.022262)
  (0.505491,nan) (0.505491,0.022483) (0.505598,0.022693) (0.505639,0.022859) (0.505641,0.022887)
  (0.505743,0.023302) (0.505743,nan) (0.505743,0.023542) (0.505859,0.023770) (0.505961,0.024184)
  (0.505964,0.024217) (0.506020,0.024444) (0.506020,nan) (0.506020,0.024706) (0.506145,0.024953)
  (0.506322,0.025672) (0.506324,0.025692) (0.506324,nan) (0.506324,0.025991) (0.506461,0.026261)
  (0.506661,0.027073) (0.506661,nan) (0.506661,0.027417) (0.506811,0.027713) (0.507035,0.028627)
  (0.507035,nan) (0.507035,0.029009) (0.507200,0.029333) (0.507454,0.030370) (0.507454,nan)
  (0.507454,0.030797) (0.507657,0.031197) (0.507728,0.031483) (0.507734,0.031556) (0.507926,0.032340)
  (0.507926,nan) (0.507926,0.032821) (0.508154,0.033268) (0.508346,0.034050) (0.508354,0.034142)
  (0.508462,0.034584) (0.508462,nan) (0.508462,0.035129) (0.508719,0.035633) (0.509071,0.037072)
  (0.509076,0.037129) (0.509076,nan) (0.509076,0.037789) (0.509367,0.038359) (0.509786,0.040075)
  (0.509786,nan) (0.509786,0.040885) (0.510117,0.041534) (0.510615,0.043575) (0.510615,nan)
  (0.510615,0.044533) (0.510996,0.045278) (0.511596,0.047749) (0.511596,nan) (0.511596,0.048898)
  (0.512126,0.049929) (0.512266,0.050505) (0.512291,0.050803) (0.512777,0.052811) (0.512777,nan)
  (0.512777,0.054214) (0.513417,0.055456) (0.513894,0.057429) (0.513931,0.057865) (0.514223,0.059081)
  (0.514223,nan) (0.514223,0.060832) (0.515012,0.062356) (0.516018,0.066535) (0.516043,0.066845)
  (0.516043,nan) (0.516043,0.069313) (0.517028,0.071207) (0.518395,0.076923) (0.518395,nan)
  (0.518395,0.080572) (0.519651,0.082969) (0.521531,0.090909) (0.521531,nan) (0.521531,0.096225)
  (0.523179,0.099338) (0.525926,0.111111) (0.525926,nan) (0.525926,0.119577) (0.529412,0.126050)
  (0.529691,0.130641) (0.532468,0.142857) (0.532468,nan) (0.532468,0.158442) (0.538462,0.169231)
  (0.540230,0.172414) (0.540984,0.191257) (0.542857,0.200000) (0.542857,nan) (0.542857,0.238095)
  (0.550000,0.250000)
};
\addplot[only marks, mark=*, mark size=0.9pt, black] coordinates {
  (0.501,0.004010) (0.502,0.008049) (0.503,0.012136) (0.504,0.016157) (0.505,0.020244)
  (0.506,0.024364) (0.507,0.028483) (0.508,0.032966) (0.509,0.036783) (0.510,0.041304)
  (0.511,0.045294) (0.512,0.049684) (0.513,0.054647) (0.514,0.058154) (0.515,0.062333)
  (0.516,0.066462) (0.517,0.071154) (0.518,0.075273) (0.519,0.081727) (0.520,0.084444)
  (0.521,0.088667) (0.522,0.097111) (0.523,0.099000) (0.524,0.102857) (0.525,0.107143)
  (0.526,0.119714) (0.527,0.121571) (0.528,0.123429) (0.529,0.125286) (0.530,0.132000)
  (0.531,0.136400) (0.532,0.140800) (0.533,0.159400) (0.534,0.161200) (0.535,0.163000)
  (0.536,0.164800) (0.537,0.166600) (0.538,0.168400) (0.539,0.170200) (0.540,0.172000)
  (0.541,0.191333) (0.542,0.196000) (0.543,0.238333) (0.544,0.240000) (0.545,0.241667)
  (0.546,0.243333) (0.547,0.245000) (0.548,0.246667) (0.549,0.248333) (0.550,0.250000)
};
\legend{$\gamma^{*}(\theta)$, values in Table 1}
\end{axis}
\end{tikzpicture}
\begin{center}
\textbf{Figure 1.} The function $\gamma^{*}(\theta)$ for $\frac{1}{2} < \theta \leqslant \frac{11}{20}$. It is piecewise linear and monotonic increasing, with infinitely many pieces accumulating at $\theta = \frac{1}{2}$. The largest jumps occur at $\theta = \frac{19}{35}, \frac{41}{77}, \frac{71}{135}, \frac{109}{209}, \frac{155}{299}$.
\end{center}

Apart from results with conditions similar to \eqref{eq:condition_gamma}, we can also get new results with different types of conditions. Here we provide an example of such results.
\begin{theorem1number}\label{t2}
Let $\frac{1}{2} < \theta < 1$. Write $M = x^{\alpha_1}$, $N = x^{\alpha_2}$ and $R = x^{\alpha_3}$. Then \eqref{eq:form1} holds if one of $M(s)$, $N(s)$, $R(s)$ is prime-factored, $\alpha_1 \leqslant \alpha_2 \leqslant \alpha_3$ and either $\alpha_2 \geqslant 1 - \theta$, or
$$
\frac{5(1 - \theta)}{8} < \alpha_2 < 1 - \theta, \quad \text{and there exists } k \in \mathbb{Z},\ 1 \leqslant k \leqslant \left\lceil \frac{1 - \theta}{\alpha_1} \right\rceil
$$
such that both of the following two inequalities hold:
$$
\min\left(1 - \theta - \frac{2}{5} \alpha_2,\ \frac{3(1 - \theta)}{2} - \alpha_2 \right) > k \left(2(1 - \theta) - \alpha_2 - \min(\alpha_3, 1 - \theta) \right),
$$
\begin{align}
\nonumber &\ \min\left(1 - \theta - \frac{2}{5} \alpha_2,\ \frac{3(1 - \theta)}{2} - \alpha_2 \right) \\
\nonumber >&\ \min\left(1 - \theta - 2 k \alpha_1 + 3 k \left(2(1 - \theta) - \alpha_2 - \min(\alpha_3, 1 - \theta) \right), \right. \\
\nonumber &\ \qquad \qquad 1 - \theta - k \alpha_1 + k \left(2(1 - \theta) - \alpha_2 - \min(\alpha_3, 1 - \theta) \right), \\
\nonumber &\ \qquad \qquad \qquad \max\left( 1 - \theta - \frac{8}{5} k \alpha_1 + 2 k \left(2(1 - \theta) - \alpha_2 - \min(\alpha_3, 1 - \theta) \right), \right. \\
\nonumber &\ \left. \left. \qquad \qquad \qquad \qquad \qquad \frac{1 - \theta}{2} - k \alpha_1 + 2 k \left(2(1 - \theta) - \alpha_2 - \min(\alpha_3, 1 - \theta) \right) \right) \right).
\end{align}
\end{theorem1number}
There are many interesting applications of our Theorems. For example, by replacing the function $\gamma_{\operatorname{BHP}}(\theta)$ used in \cite{LRB052} with our new $\gamma^{*}(\theta)$ in Theorem~\ref{t1}, we can improve the exponent $\theta = 0.52$ concerning primes in short intervals obtained in \cite{LRB052}. The upper and lower bounds for the functions $\mathbf{LB}(\theta)$ and $\mathbf{UB}(\theta)$ defined in \cite{LRB052} can also be improved.
\begin{theorem1number}\label{t3}
There exists a computable number $c > 0$ such that, for all sufficiently large $x$, the interval $[x - x^{0.52 - c}, x]$ contains prime numbers.
\end{theorem1number}

\section{The Matom\"{a}ki--Pandey--Pliego--Ter\"{a}v\"{a}inen--Wang Method}
In this section we briefly introduce the new method of obtaining \eqref{eq:form1} developed by Matom\"{a}ki, Pandey, Pliego, Ter\"{a}v\"{a}inen and Wang \cite{MPPTW} and state some results we need in the proof of Theorem~\ref{t1} and Theorem~\ref{t2}.
Let $\mathcal{T} \subseteq [T, 2T]$. Define $\mathcal{N}_M(\nu; \mathcal{T})$ as the maximal cardinality of a $1$-separated subset $\mathcal{U} \subseteq \mathcal{T}$ such that
$$
\left|M \left( \frac{1}{2}+it \right)\right| \geqslant M^{\frac{1}{2} - \nu} \quad \text{for all } t \in \mathcal{U}.
$$
By applying the classical mean value estimate [\cite{IwaniecKowalski}, Theorem 9.7], Halász--Montgomery--Huxley large value estimate [\cite{IwaniecKowalski}, Corollary 9.9], Guth--Maynard large value estimate [\cite{GuthMaynard}, Theorem 1.1] and Chen--Gupta--Li large value estimate [\cite{ChenGuptaLi}, Theorem 1.1] to $M^k(s)$, Matom\"{a}ki, Pandey, Pliego, Ter\"{a}v\"{a}inen and Wang [\cite{MPPTW}, Lemma 2.3(i)] obtained the following upper bound for $\mathcal{N}_M(\nu; \mathcal{T})$ that holds for every fixed $k \in \mathbb{N}$:
\begin{equation}
\mathcal{N}_M(\nu; \mathcal{T}) \ll x^{\mathcal{K}(k \alpha, \nu; \theta) + o(1)},
\end{equation}
where
\begin{equation}
\mathcal{K}(\alpha, \nu; \theta) = \min\left(1 - \theta,\ \max\left(2 \alpha \nu,\ \mathcal{K}_{0}(\alpha, \nu; \theta) \right) \right),
\end{equation}
\begin{align}
\nonumber \mathcal{K}_{0}(\alpha, \nu; \theta) =&\ \min\left(1 - \theta + \alpha (6 \nu - 2),\ 1 - \theta + \alpha (2 \nu - 1), \right. \\
& \left. \qquad \qquad \max\left(1 - \theta + \alpha \left(4 \nu - \frac{8}{5}\right),\ \frac{1 - \theta}{2} + \alpha (4 \nu - 1) \right) \right).
\end{align}

\begin{lemma}\label{MPPTW_Proposition3.2}
Let $M_j = x^{\alpha_j}$. Suppose that $\alpha_1 + \alpha_2 + \alpha_3 = 1$ and $\alpha_1, \alpha_2, \alpha_3 \geqslant \varepsilon$. Suppose also that one of $M_1(s)$, $M_2(s)$ and $M_3(s)$ is prime-factored. For $j = 1, 2, 3$, define the function $P_j : [0, 1 - \theta] \to [0, \infty]$ as
$$
P_j(t) = \inf \left\{\alpha_j \nu : \nu \in [\varepsilon^{6}, \varepsilon^{-1}], \min_{1 \leqslant k \leqslant \lceil \varepsilon^{-7} \rceil} \mathcal{K}(k \alpha_j, \nu; \theta) \geqslant t \right\}.
$$
If we have
$$
P_1(t) + P_2(t) + P_3(t) > t + \varepsilon^{10}
$$
for every $t \in [0, 1 - \theta]$, then we have
$$
\int_{\mathcal{T}} \left|M_1 \left( \frac{1}{2}+it \right) M_2 \left( \frac{1}{2}+it \right) M_3 \left( \frac{1}{2}+it \right)\right| \ll x^{\frac{1}{2}} (\log x)^{-K}.
$$
\end{lemma}
\begin{proof}
This is [\cite{MPPTW}, Proposition 3.2] with $J = 3$, $\tau = 1 - \theta$ and 
$$
H_j(\nu) = \min_{1 \leqslant k \leqslant \lceil \varepsilon^{-7} \rceil} \mathcal{K}(k \alpha_j, \nu; \theta).
$$
\end{proof}

Let
\begin{equation}\label{eq:f}
f(t) = P_1(t) + P_2(t) + P_3(t) - t.
\end{equation}
By a standard dyadic partition (note that $\mathcal{R}_{\alpha}(t)$, which will be defined later, is non-decreasing on $\theta$, considering $\max \theta$ suffices) together with Lemma~\ref{MPPTW_Proposition3.2}, we can reduce the problem of proving \eqref{eq:form1} to the task of showing a piecewise-linear function $f(t) - \varepsilon^{10}$ is positive. In order to give a lower bound for $P_j(t)$, we define the following auxiliary function $\mathcal{R}_{\alpha}(t)$ as
\begin{equation}
\mathcal{R}_{\alpha}(t) = \alpha \inf \left\{\nu : \nu \geqslant 0,\ \mathcal{K}(\alpha, \nu; \theta) \geqslant t \right\}.
\end{equation}
Since $\mathcal{K}(\alpha, \nu; \theta)$ has an explicit expression, we can also write $\mathcal{R}_{\alpha}(t)$ explicitly as
\begin{align}\label{eq:R_alpha}
\nonumber \mathcal{R}_{\alpha}(t) =&\ \max\left(0,\ \min\left(\frac{t}{2},\ \max\left(\frac{t - (1 - \theta) + 2 \alpha}{6},\ \frac{t - (1 - \theta) + \alpha}{2}, \right. \right. \right. \\
& \qquad \qquad \qquad \qquad \qquad \qquad \left. \left. \left. \min\left(\frac{t - (1 - \theta) + \frac{8}{5} \alpha}{4},\ \frac{t - \frac{1 - \theta}{2} + \alpha}{4} \right) \right) \right) \right).
\end{align}
Moreover, we have
\begin{align}\label{eq:R_alpha_explicit}
\mathcal{R}_{\alpha}(t) =
\begin{cases}
0, & \quad \varepsilon \leqslant \alpha \leqslant \frac{1 - \theta}{2} \text{ and } 0 \leqslant t \leqslant 1 - \theta - 2 \alpha, \\
\frac{t - (1 - \theta) + 2 \alpha}{6}, & \quad \varepsilon \leqslant \alpha \leqslant \frac{1 - \theta}{2} \text{ and } 1 - \theta - 2 \alpha \leqslant t \leqslant 1 - \theta - \frac{4}{5} \alpha, \\
\frac{t - (1 - \theta) + \frac{8}{5} \alpha}{4}, & \quad \varepsilon \leqslant \alpha \leqslant \frac{1 - \theta}{2} \text{ and } 1 - \theta - \frac{4}{5} \alpha \leqslant t \leqslant 1 - \theta - \frac{2}{5} \alpha, \\
\frac{t - (1 - \theta) + \alpha}{2}, & \quad \varepsilon \leqslant \alpha \leqslant \frac{1 - \theta}{2} \text{ and } 1 - \theta - \frac{2}{5} \alpha \leqslant t \leqslant 1 - \theta, \\
\frac{t}{2}, & \quad \frac{1 - \theta}{2} \leqslant \alpha \leqslant \frac{5(1 - \theta)}{6} \text{ and } 0 \leqslant t \leqslant \alpha - \frac{1 - \theta}{2}, \\
\frac{t - (1 - \theta) + 2 \alpha}{6}, & \quad \frac{1 - \theta}{2} \leqslant \alpha \leqslant \frac{5(1 - \theta)}{6} \text{ and } \alpha - \frac{1 - \theta}{2} \leqslant t \leqslant 1 - \theta - \frac{4}{5} \alpha, \\
\frac{t - (1 - \theta) + \frac{8}{5} \alpha}{4}, & \quad \frac{1 - \theta}{2} \leqslant \alpha \leqslant \frac{5(1 - \theta)}{6} \text{ and } 1 - \theta - \frac{4}{5} \alpha \leqslant t \leqslant 1 - \theta - \frac{2}{5} \alpha, \\
\frac{t - (1 - \theta) + \alpha}{2}, & \quad \frac{1 - \theta}{2} \leqslant \alpha \leqslant \frac{5(1 - \theta)}{6} \text{ and } 1 - \theta - \frac{2}{5} \alpha \leqslant t \leqslant 1 - \theta, \\
\frac{t}{2}, & \quad \frac{5(1 - \theta)}{6} \leqslant \alpha \leqslant 1 - \theta \text{ and } 0 \leqslant t \leqslant \alpha - \frac{1 - \theta}{2}, \\
\frac{t - \frac{1 - \theta}{2} + \alpha}{4}, & \quad \frac{5(1 - \theta)}{6} \leqslant \alpha \leqslant 1 - \theta \text{ and } \alpha - \frac{1 - \theta}{2} \leqslant t \leqslant \frac{3(1 - \theta)}{2} - \alpha, \\
\frac{t - (1 - \theta) + \alpha}{2}, & \quad \frac{5(1 - \theta)}{6} \leqslant \alpha \leqslant 1 - \theta \text{ and } \frac{3(1 - \theta)}{2} - \alpha \leqslant t \leqslant 1 - \theta, \\
\frac{t}{2}, & \quad \alpha > 1 - \theta. \\
\end{cases}
\end{align}
Now, by the definitions of $P_j(t)$ and $\mathcal{R}_{\alpha}(t)$, we can easily deduce that
\begin{equation}\label{eq:P_j}
P_j(t) = \max_{1 \leqslant k \leqslant \lceil \varepsilon^{-7} \rceil} \left(\frac{1}{k} \mathcal{R}_{k \alpha_j}(t) \right).
\end{equation}
Thus, we can use $\mathcal{R}_{\alpha}(t)$ to give lower bounds for $P_j(t)$. We also note that for $t \leqslant \frac{\varepsilon^6}{2}$, we have $P_1(t) + P_2(t) + P_3(t) \geqslant (\alpha_1 + \alpha_2 + \alpha_3) \varepsilon^6 = \varepsilon^6$ and $f(t) - \varepsilon^{10} \geqslant \frac{\varepsilon^6}{2} - \varepsilon^{10} > 0$. Hence we can assume that $t > \frac{\varepsilon^6}{2}$.

For $0 < \alpha \leqslant 1 - \theta$, we define
\begin{equation}\label{eq:w-phi}
w(\alpha) = \min\left(1 - \theta - \frac{2}{5} \alpha,\ \frac{3(1 - \theta)}{2} - \alpha \right),\ \phi(\alpha) = \min\left(\frac{2}{5} \alpha - \frac{1 - \theta}{4},\ \frac{\alpha}{4} - \frac{1 - \theta}{8} \right).
\end{equation}
It is easy to see that
\begin{align*}
w(\alpha) =
\begin{cases}
1 - \theta - \frac{2}{5} \alpha, & \quad \alpha \leqslant \frac{5(1 - \theta)}{6}, \\
\frac{3(1 - \theta)}{2} - \alpha, & \quad \alpha \geqslant \frac{5(1 - \theta)}{6}, \\
\end{cases}
\qquad \phi(\alpha) =
\begin{cases}
\frac{2}{5} \alpha - \frac{1 - \theta}{4}, & \quad \alpha \leqslant \frac{5(1 - \theta)}{6}, \\
\frac{\alpha}{4} - \frac{1 - \theta}{8}, & \quad \alpha \geqslant \frac{5(1 - \theta)}{6}. \\
\end{cases}
\end{align*}
Moreover, we have
\begin{equation}\label{eq:phi5/8}
\phi\left(\frac{5(1 - \theta)}{8} + \delta \right) \geqslant \min\left(\frac{2 \delta}{5},\ \frac{1 - \theta}{12} \right)
\end{equation}
for any $\delta \in [0, \frac{3(1 - \theta)}{8}]$, since $\phi(\frac{5(1 - \theta)}{8} + \delta) = \frac{2 \delta}{5}$ if $\frac{5(1 - \theta)}{8} + \delta \leqslant \frac{5(1 - \theta)}{6}$ and $\phi(\alpha) \geqslant \phi(\frac{5(1 - \theta)}{6}) = \frac{1 - \theta}{12}$ if $\alpha > \frac{5(1 - \theta)}{6}$.

\begin{lemma}\label{l22}
Let $0 < \alpha \leqslant 1 - \theta$ and $0 \leqslant t \leqslant 1 - \theta$. Then
$$
\mathcal{R}_{\alpha}(t) \geqslant \frac{t}{4} + \min\left(\frac{t}{4},\ \phi(\alpha) \right) + \frac{1}{4} \max(t - w(\alpha), 0).
$$
Moreover, we have
$$
\phi(\alpha) \leqslant \frac{1 - \theta}{8} \leqslant \frac{w(\alpha^{\prime})}{4} \quad (0 < \alpha, \alpha^{\prime} \leqslant 1 - \theta), \qquad \phi(\alpha) + \frac{1 - \theta - w(\alpha)}{4} = \frac{\alpha}{2} - \frac{1 - \theta}{4}.
$$
\end{lemma}
\begin{proof}
By \eqref{eq:w-phi}, we have 
$$
\min\left(\frac{t - (1 - \theta) + \frac{8}{5} \alpha}{4},\ \frac{t - \frac{1 - \theta}{2} + \alpha}{4} \right) = \frac{t}{4} + \phi(\alpha),
$$
and
$$
\frac{t - (1 - \theta) + \alpha}{2} = \frac{t}{2} + \phi(\alpha) - \frac{w(\alpha)}{4}. \qquad \text{(this proves the last identity)}
$$
Hence
\begin{align}
\nonumber &\ \max\left(\frac{t - (1 - \theta) + \alpha}{2}, \min\left(\frac{t - (1 - \theta) + \frac{8}{5} \alpha}{4},\ \frac{t - \frac{1 - \theta}{2} + \alpha}{4} \right) \right) \\
\nonumber \geqslant&\ \frac{t}{4} + \phi(\alpha) + \frac{1}{4} \max(t - w(\alpha), 0).
\end{align}
By \eqref{eq:R_alpha}, we have
\begin{align}
\nonumber \mathcal{R}_{\alpha}(t) \geqslant&\ \min\left(\frac{t}{2},\ \max\left(\frac{t - (1 - \theta) + \alpha}{2}, \min\left(\frac{t - (1 - \theta) + \frac{8}{5} \alpha}{4},\ \frac{t - \frac{1 - \theta}{2} + \alpha}{4} \right) \right) \right) \\
\nonumber =&\ \min\left(\frac{t}{2},\ \frac{t}{4} + \phi(\alpha) + \frac{1}{4} \max(t - w(\alpha), 0) \right) \\
\nonumber =&\ \frac{t}{4} + \min\left(\frac{t}{4},\ \phi(\alpha) + \frac{1}{4} \max(t - w(\alpha), 0) \right).
\end{align}

Since $\alpha, \alpha^{\prime} \leqslant 1 - \theta$, we have $\phi(\alpha) \leqslant \frac{\alpha}{4} - \frac{1 - \theta}{8} \leqslant \frac{1 - \theta}{8}$ and $w(\alpha^{\prime}) \geqslant \min(\frac{3(1 - \theta)}{5},\ \frac{1 - \theta}{2}) = \frac{1 - \theta}{2}$, and thus
$$
\phi(\alpha) \leqslant \frac{1 - \theta}{8} \leqslant \frac{w(\alpha^{\prime})}{4}.
$$
If $t \leqslant w(\alpha)$, then the first part of Lemma~\ref{l22} holds. If $t > w(\alpha)$, then
$$
\phi(\alpha) + \frac{\max(t - w(\alpha), 0)}{4} = \phi(\alpha) + \frac{t - w(\alpha)}{4} \leqslant \frac{w(\alpha)}{4} + \frac{t - w(\alpha)}{4} = \frac{t}{4},
$$
and
$$
\min\left(\frac{t}{4},\ \phi(\alpha) + \frac{1}{4} \max(t - w(\alpha), 0) \right) = \phi(\alpha) + \frac{1}{4} \max(t - w(\alpha), 0).
$$
Hence the first part of Lemma~\ref{l22} also holds.
\end{proof}

\begin{lemma}\label{l23}
For $k \in \mathbb{N}$, write
$$
l_k(\theta) = 2 k - (4 k - 1) \theta \quad \text{and} \quad a_k(\theta) = \frac{(12 k - 1) \theta - 6 k}{2 k - 1}.
$$
Suppose that $2 \theta - 1 < \alpha_3 \leqslant \frac{5 \theta - 2}{3}$. Then we have
$$
\frac{2 - 3 \theta + \alpha_3}{4} - k \frac{\alpha_3 - (2 \theta - 1)}{2} = \frac{(2 k - 1)(u_k(\theta) - \alpha_3)}{4},
$$
$$
\frac{\frac{2 - 3 \theta + \alpha_3}{2} - (1 - \theta) + 2 k \alpha_3}{6} - k \frac{\alpha_3 - (2 \theta - 1)}{2} = \frac{(2 k - 1)(a_k(\theta) - \alpha_3)}{12},
$$
$$
\frac{\frac{2 - 3 \theta + \alpha_3}{2} - (1 - \theta) + k \alpha_3}{2} - k \frac{\alpha_3 - (2 \theta - 1)}{2} = \frac{\alpha_3 - l_k(\theta)}{4},
$$
$$
\frac{\frac{2 - 3 \theta + \alpha_3}{2} - \frac{1 - \theta}{2} + k \alpha_3}{4} - k \frac{\alpha_3 - (2 \theta - 1)}{2} = \frac{(2 k - 1)(q_k(\theta) - \alpha_3)}{8},
$$
$$
\frac{\frac{2 - 3 \theta + \alpha_3}{2} - (1 - \theta) + \frac{8}{5} k \alpha_3}{4} - k \frac{\alpha_3 - (2 \theta - 1)}{2} = \frac{(4 k - 5)(p_k(\theta) - \alpha_3)}{40}, \qquad (k \geqslant 2)
$$
$$
\frac{\frac{2 - 3 \theta + \alpha_3}{2} - (1 - \theta) + \frac{8}{5} \alpha_3}{4} - \frac{\alpha_3 - (2 \theta - 1)}{2} = \frac{\alpha_3 - (20 - 35 \theta)}{40}. \qquad (k = 1)
$$
Consequently, $\frac{1}{k} \mathcal{R}_{k \alpha_3}(\frac{2 - 3 \theta + \alpha_3}{2}) > \frac{\alpha_3 - (2 \theta - 1)}{2} + \varepsilon^7$ if $\alpha_3 < u_k(\theta) - \varepsilon^6$ and at least one of the following holds:

$\alpha_3 < a_k(\theta) - \varepsilon^6$;

$\alpha_3 > l_k(\theta) + \varepsilon^6$;

$\frac{\frac{2 - 3 \theta + \alpha_3}{2} - (1 - \theta) + \frac{8}{5} k \alpha_3}{4} > k \frac{\alpha_3 - (2 \theta - 1)}{2} + \varepsilon^6$ and $\alpha_3 < q_k(\theta) - \varepsilon^6$.
\end{lemma}
\begin{proof}
The identities in this lemma can be easily verified. The last assertion follows from \eqref{eq:R_alpha}, since $k \frac{\alpha_3 - (2 \theta - 1)}{2} > 0$.
\end{proof}

\begin{lemma}\label{l24}
Let $k \geqslant 2$. Then

(i). $q_k(\theta) < u_k(\theta)$ if and only if $2 \theta - 1 < \frac{1}{4 k + 1}$;

(ii). $q_k(\theta) \leqslant p_k(\theta)$ if and only if $2 \theta - 1 \geqslant \frac{10 k - 5}{48 k^2 - 2 k - 5}$;

(iii). $u_k(\theta) \leqslant p_k(\theta)$ if and only if $2 \theta - 1 \geqslant \frac{7 k - 5}{32 k^2 - 9 k - 5}$;

(iv). $4 \theta - 2 \leqslant u_k(\theta)$ if and only if $2 \theta - 1 \leqslant \frac{1}{4 k - 1}$;

(v). $4 \theta - 2 \leqslant p_k(\theta)$ if and only if $2 \theta - 1 \geqslant \frac{5}{24 k + 15}$;

(vi). $l_k(\theta) < q_k(\theta)$ if and only if $2 \theta - 1 > \frac{2 k - 1}{(4 k - 1)(2 k + 1)}$;

(vii). $l_{k - 1}(\theta) < u_k(\theta)$ if and only if $2 \theta - 1 > \frac{1}{4 k - 1}$.

Moreover, for $\frac{1}{2} < \theta < \frac{4}{7}$ we have
$$
4 \theta - 2 \leqslant \Gamma_0(\theta) \leqslant u_2(\theta) = \frac{5 \theta - 2}{3},
$$
and for $\frac{5}{9} \leqslant \theta < \frac{4}{7}$ we have
$$
\gamma^{*}(\theta) = \Gamma_0(\theta) = u_2(\theta) = \frac{5 \theta - 2}{3}.
$$
\end{lemma}
\begin{proof}
Each of (i)--(vii) is a linear inequality on $\theta$ and can be directly verified. Let $\frac{1}{2} < \theta < \frac{4}{7}$, then $0 < 2 \theta - 1 < \frac{1}{7}$. The intervals $[\frac{5}{24 k + 15},\ \frac{1}{4 k - 1}]$ ($k \geqslant 2$) cover $(0,\ \frac{1}{7}]$ since $\frac{5}{24 k + 15} < \frac{1}{4 k + 3}$ and $\frac{5}{24 k + 15} \to 0$. Note that $q_k(\theta) = 4 \theta - 2 + \frac{2 \theta - 1}{2 k - 1} > 4 \theta - 2$, by (iv) and (v) we know that there exists $k \geqslant 2$ such that $\min(u_k(\theta), p_k(\theta), q_k(\theta)) \geqslant 4 \theta - 2$. We also have $\min(u_k(\theta), p_k(\theta), q_k(\theta)) \leqslant u_k(\theta) \leqslant u_2(\theta)$. If $\theta \geqslant \frac{5}{9}$, then $2 \theta - 1 \geqslant \frac{1}{9}$. By (i) and (iii), we have $q_k(\theta) \geqslant u_k(\theta)$ and $p_k(\theta) \geqslant u_k(\theta)$, hence $\Gamma_0(\theta) = u_2(\theta) = \frac{5 \theta - 2}{3}$. Note that $k_0(\theta)$ is the least integer $k \geqslant 2$ such that $l_k(\theta) < \Gamma_0(\theta)$. Now $l_2(\theta) = 4 - 7 \theta < u_2(\theta)$ gives $k_0 = 2$ and $\Gamma_0(\theta) = u_2(\theta) = \gamma^{*}(\theta)$, since $\theta > \frac{7}{13}$.
\end{proof}

\section{Computational Lemmas}
In this section we shall prove some lemmas that will be helpful in the proof of our Theorem~\ref{t1} and Theorem~\ref{t2}.
\begin{lemma}\label{l31}
Suppose that $\alpha_1 + \alpha_2 + \alpha_3 = 1$, $\varepsilon \leqslant \alpha_1 \leqslant \alpha_2 \leqslant \alpha_3 \leqslant 1 - \theta$ and $\phi(\alpha_j) \geqslant \frac{\varepsilon^6}{8}$ for $j = 1, 2, 3$. Then for every $t \in [\frac{\varepsilon^6}{2},\ 1 - \theta]$, we have
$$
f(t) \geqslant \min\left(\frac{\varepsilon^6}{4},\ \phi(\alpha_1) + \phi(\alpha_2) + \phi(\alpha_3) - \frac{w(\alpha_3)}{4} \right).
$$
\end{lemma}
\begin{proof}
By Lemma~\ref{l22} and \eqref{eq:P_j}, we have
$$
f(t) \geqslant \sum_{j = 1}^{3} \mathcal{R}_{\alpha_j}(t) - t \geqslant \sum_{j = 1}^{3} \left(\min\left(\frac{t}{4},\ \phi(\alpha_j) \right) + \frac{1}{4} \max(t - w(\alpha_j), 0) \right) - \frac{t}{4} = F_1(t).
$$
It is easy to see that $F_1(t)$ is continuous and piecewise-linear. By Lemma~\ref{l22} we know that $4 \phi(\alpha) \leqslant w(\alpha^{\prime})$ for any $0 < \alpha, \alpha^{\prime} \leqslant 1 - \theta$, hence we have
$$
\frac{\varepsilon^6}{2} \leqslant 4 \phi(\alpha_1) \leqslant 4 \phi(\alpha_2) \leqslant 4 \phi(\alpha_3) \leqslant w(\alpha_3) \leqslant w(\alpha_2) \leqslant w(\alpha_1) \leqslant 1 - \theta
$$
since $\phi(\alpha)$ is increasing and $w(\alpha)$ is decreasing. The slopes of $F_1(t)$ in the 7 intervals, given by $\frac{\varepsilon^6}{2}$, $1 - \theta$ and the above 6 intermediate values, are
$$
\frac{1}{2}, \quad \frac{1}{4}, \quad 0, \quad -\frac{1}{4}, \quad 0, \quad \frac{1}{4}, \quad \frac{1}{2}
$$
respectively. Hence $F_1(t)$ is non-decreasing in $[\frac{\varepsilon^6}{2},\ 4 \phi(\alpha_3)]$, non-increasing in $[4 \phi(\alpha_3),\ w(\alpha_3)]$ and non-decreasing again in $[w(\alpha_3),\ 1 - \theta]$. Since $F_1(\frac{\varepsilon^6}{2}) = \frac{\varepsilon^6}{4}$, we only need to check the value of $F_1(w(\alpha_3))$. By Lemma~\ref{l22}, we have $\min\left(\frac{w(\alpha_j)}{4},\ \phi(\alpha_j) \right) = \phi(\alpha_j)$ for $j = 1, 2, 3$, hence
$$
F_1(w(\alpha_j)) = \phi(\alpha_1) + \phi(\alpha_2) + \phi(\alpha_3) - \frac{w(\alpha_j)}{4} + \frac{1}{4} \sum_{i} \max(w(\alpha_j) - w(\alpha_i), 0).
$$
Since $w(\alpha_3) \leqslant w(\alpha_2) \leqslant w(\alpha_1)$, we have
$$
F_1(w(\alpha_3)) = \phi(\alpha_1) + \phi(\alpha_2) + \phi(\alpha_3) - \frac{w(\alpha_3)}{4},
$$
and the proof of Lemma~\ref{l31} is completed.
\end{proof}
\begin{lemma}\label{l32}
Suppose that $\alpha_1 + \alpha_2 + \alpha_3 = 1$ and $\frac{5(1 - \theta)}{8} \leqslant \alpha_1 \leqslant \alpha_2 \leqslant \alpha_3 \leqslant 1 - \theta$.

(a). If $\alpha_2 \leqslant \frac{5(1 - \theta)}{6}$, then
$$
\phi(\alpha_1) + \phi(\alpha_2) + \phi(\alpha_3) - \frac{w(\alpha_3)}{4} = \frac{\alpha_3 - (6 - 10 \theta)}{10}.
$$

(b). If $\alpha_1 \leqslant \frac{5(1 - \theta)}{6} < \alpha_2 \leqslant \alpha_3$, then
$$
\phi(\alpha_1) + \phi(\alpha_2) + \phi(\alpha_3) - \frac{w(\alpha_3)}{4} \geqslant \frac{3}{8} - \frac{55}{64}(1 - \theta).
$$

(c). If $\alpha_1 > \frac{5(1 - \theta)}{6}$, then
$$
\phi(\alpha_1) + \phi(\alpha_2) + \phi(\alpha_3) - \frac{w(\alpha_3)}{4} = \frac{1 - 3(1-\theta) + \alpha_3}{4} \geqslant \frac{1}{3} - \frac{3(1-\theta)}{4}.
$$
\end{lemma}
\begin{proof}
(a). If $\alpha_3 \leqslant \frac{5(1 - \theta)}{6}$, we have $\phi(\alpha_1) + \phi(\alpha_2) + \phi(\alpha_3) = \frac{2}{5} - \frac{3(1 - \theta)}{4}$ and $\frac{w(\alpha_3)}{4} = \frac{1 - \theta}{4} - \frac{\alpha_3}{10}$, hence
$$
\phi(\alpha_1) + \phi(\alpha_2) + \phi(\alpha_3) - \frac{w(\alpha_3)}{4} = \frac{2}{5} - (1 - \theta) + \frac{\alpha_3}{10} = \frac{\alpha_3 - (6 - 10 \theta)}{10}.
$$
If $\alpha_3 > \frac{5(1 - \theta)}{6}$, we have
\begin{align}
\nonumber &\ \phi(\alpha_1) + \phi(\alpha_2) + \phi(\alpha_3) - \frac{w(\alpha_3)}{4} \\
\nonumber =&\ \left(\frac{\alpha_3}{4} - \frac{1 - \theta}{8} \right) + \left(\frac{2}{5}(1 - \alpha_3) - \frac{1 - \theta}{2} \right) - \frac{3(1 - \theta)}{8} + \frac{\alpha_3}{4} \\
\nonumber =&\ \frac{\alpha_3 - (6 - 10 \theta)}{10}.
\end{align}

(b). We have
\begin{align}
\nonumber &\ \phi(\alpha_1) + \phi(\alpha_2) + \phi(\alpha_3) - \frac{w(\alpha_3)}{4} \\
\nonumber =&\ \left(\frac{\alpha_2 + \alpha_3 - (1 - \theta)}{4} \right) + \left(\frac{2 (1 - \alpha_2 - \alpha_3)}{5} - \frac{1 - \theta}{4} \right) - \frac{3(1 - \theta)}{8} + \frac{\alpha_3}{4} \\
\nonumber =&\ \frac{2}{5} - \frac{7(1 - \theta)}{8} + \frac{\alpha_3}{10} - \frac{3 \alpha_2}{20}.
\end{align}
Since $\alpha_2 \leqslant \alpha_3$ and $\alpha_2 \leqslant \frac{\alpha_2 + \alpha_3}{2} = \frac{1 - \alpha_1}{2} \leqslant \frac{1}{2} - \frac{5(1 - \theta)}{16}$, we have
\begin{align}
\nonumber \phi(\alpha_1) + \phi(\alpha_2) + \phi(\alpha_3) - \frac{w(\alpha_3)}{4} =&\ \frac{2}{5} - \frac{7(1 - \theta)}{8} + \frac{\alpha_3}{10} - \frac{3 \alpha_2}{20} \\
\nonumber \geqslant&\ \frac{2}{5} - \frac{7(1 - \theta)}{8} - \frac{\alpha_2}{20} \\
\nonumber \geqslant&\ \frac{2}{5} - \frac{7(1 - \theta)}{8} - \frac{1}{40} + \frac{1 - \theta}{64} \\
\nonumber =&\ \frac{3}{8} - \frac{55}{64}(1 - \theta).
\end{align}

(c). We have
$$
\phi(\alpha_1) + \phi(\alpha_2) + \phi(\alpha_3) - \frac{w(\alpha_3)}{4} = \frac{1}{4} + \frac{3(1 - \theta)}{8} - \frac{3(1 - \theta)}{8} + \frac{\alpha_3}{4} = \frac{1 - 3(1-\theta) + \alpha_3}{4}.
$$
Since $\alpha_3 \geqslant \frac{1}{3}$, this is $\geqslant \frac{1}{3} - \frac{3(1-\theta)}{4}$.
\end{proof}
\begin{lemma}\label{l33}
Suppose that $\alpha_1 + \alpha_2 + \alpha_3 = 1$, $\alpha_1 \geqslant \frac{1 - \theta}{2}$, $\varepsilon \leqslant \alpha_2, \alpha_3 \leqslant 1 - \theta$ and $\phi(\alpha_j) \geqslant \frac{\varepsilon^6}{8}$ for $j = 2, 3$. Then for every $t \in [\frac{\varepsilon^6}{2},\ 1 - \theta]$, we have
$$
f(t) \geqslant \min\left(\frac{\varepsilon^6}{8},\ \phi(\alpha_2) + \phi(\alpha_3) - \frac{w(\max(\alpha_2, \alpha_3))}{4} \right).
$$
\end{lemma}
\begin{proof}
The proof is similar to the proof of Lemma~\ref{l31}. By \eqref{eq:R_alpha_explicit} we know that $P_1(t) \geqslant \frac{1}{2} \mathcal{R}_{2 \alpha_1}(t) = \frac{t}{4}$. By Lemma~\ref{l22}, we have
$$
f(t) \geqslant \frac{t}{4} + \sum_{j = 2}^{3} \mathcal{R}_{\alpha_j}(t) - t \geqslant \sum_{j = 2}^{3} \left(\min\left(\frac{t}{4},\ \phi(\alpha_j) \right) + \frac{1}{4} \max(t - w(\alpha_j), 0) \right) - \frac{t}{4} = F_2(t).
$$
It is easy to see that $F_2(t)$ is continuous and piecewise-linear. By Lemma~\ref{l22} we know that $4 \phi(\alpha) \leqslant w(\alpha^{\prime})$ for any $0 < \alpha, \alpha^{\prime} \leqslant 1 - \theta$, hence we have
$$
\frac{\varepsilon^6}{2} \leqslant 4 \phi(\min(\alpha_2, \alpha_3)) \leqslant 4 \phi(\max(\alpha_2, \alpha_3)) \leqslant w(\max(\alpha_2, \alpha_3)) \leqslant w(\min(\alpha_2, \alpha_3)) \leqslant 1 - \theta
$$
since $\phi(\alpha)$ is increasing and $w(\alpha)$ is decreasing. The slopes of $F_2(t)$ in the 5 intervals, given by $\frac{\varepsilon^6}{2}$, $1 - \theta$ and the above 4 intermediate values, are
$$
\frac{1}{4}, \quad 0, \quad -\frac{1}{4}, \quad 0, \quad \frac{1}{4}
$$
respectively. Hence $F_2(t)$ is non-decreasing in $[\frac{\varepsilon^6}{2},\ 4 \phi(\max(\alpha_2, \alpha_3))]$, non-increasing in 

\noindent $[4 \phi(\max(\alpha_2, \alpha_3)),\ w(\max(\alpha_2, \alpha_3))]$ and non-decreasing again in $[w(\max(\alpha_2, \alpha_3)),\ 1 - \theta]$. Since $F_2(\frac{\varepsilon^6}{2}) = \frac{\varepsilon^6}{8}$, we only need to check the value of $F_2(w(\max(\alpha_2, \alpha_3)))$. Since we have
$$
F_2(w(\max(\alpha_2, \alpha_3))) \geqslant \phi(\alpha_2) + \phi(\alpha_3) - \frac{w(\max(\alpha_2, \alpha_3))}{4},
$$
the proof of Lemma~\ref{l33} is completed.
\end{proof}

\begin{lemma}\label{l34}
Suppose that $\alpha_1 + \alpha_2 + \alpha_3 = 1$, $\alpha_1, \alpha_2, \alpha_3 \geqslant \varepsilon$ and
$$
\frac{5(1 - \theta)}{8} \leqslant \alpha_2 \leqslant \min(\alpha_3, 1 - \theta).
$$
Assume that there exists an integer $k \geqslant 2$ such that
$$
\frac{1}{k} \mathcal{R}_{k \alpha_1}(w(\alpha_2)) > \frac{2(1 - \theta) - \alpha_2 - \min(\alpha_3, 1 - \theta)}{2} + \varepsilon^7.
$$
Then we have $f(t) > \varepsilon^{10}$ for every $t \in [\frac{\varepsilon^6}{2},\ 1 - \theta]$. Moreover, we have $f(t) > \varepsilon^{10}$ for every $t \in [\frac{\varepsilon^6}{2},\ 1 - \theta]$ if $\alpha_2, \alpha_3 \geqslant 1 - \theta$.
\end{lemma}
\begin{proof}
Assume first that $\frac{5(1 - \theta)}{8} \leqslant \alpha_2 \leqslant \alpha_3 \leqslant 1 - \theta$. By \eqref{eq:f} and \eqref{eq:P_j}, we have
$$
f(t) \geqslant \frac{1}{k} \mathcal{R}_{k \alpha_1}(t) + \mathcal{R}_{\alpha_2}(t) + \mathcal{R}_{\alpha_3}(t) - t = F_3(t).
$$
It is easy to see that $F_3(t)$ is continuous and piecewise-linear. For $t \in [\frac{\varepsilon^6}{2},\ \alpha_2 - \frac{1 - \theta}{2}]$, we have $\mathcal{R}_{\alpha_2}(t) = \mathcal{R}_{\alpha_3}(t) = \frac{t}{2}$, and by \eqref{eq:f} we know that
$$
f(t) = P_1(t) \geqslant \frac{1}{\lceil \frac{1 - \theta}{\varepsilon} \rceil} \mathcal{R}_{\lceil \frac{1 - \theta}{\varepsilon} \rceil \alpha_1}(t) = \frac{t}{2 \lceil \frac{1 - \theta}{\varepsilon} \rceil} > \varepsilon^8 > \varepsilon^{10}
$$
since $\alpha_1 \geqslant \varepsilon$. Hence $f(t) > \varepsilon^{10}$ for $t \in [\frac{\varepsilon^6}{2},\ \alpha_2 - \frac{1 - \theta}{2}]$. For $t \in [\alpha_2 - \frac{1 - \theta}{2},\ w(\alpha_2)]$, the slope of $\mathcal{R}_{\alpha_2}(t)$ is $\leqslant \frac{1}{4}$, the slope of $\mathcal{R}_{\alpha_3}(t)$ is $\leqslant \frac{1}{2}$, and the slope of $\frac{1}{k} \mathcal{R}_{k \alpha_1}(t)$ is $\leqslant \frac{1}{2 k} \leqslant \frac{1}{4}$, hence $F_3(t)$ is non-increasing in $[\alpha_2 - \frac{1 - \theta}{2},\ w(\alpha_2)]$. For $t \in [w(\alpha_2),\ 1 - \theta]$, both $\mathcal{R}_{\alpha_2}(t)$ and $\mathcal{R}_{\alpha_3}(t)$ have slope $\frac{1}{2}$, hence $F_3(t)$ is non-decreasing in $[w(\alpha_2),\ 1 - \theta]$. Thus, we only need to check the value of $F_3(w(\alpha_2))$. Now,
\begin{align}
\nonumber F_3(w(\alpha_2)) =&\ \frac{1}{k} \mathcal{R}_{k \alpha_1}(w(\alpha_2)) + \mathcal{R}_{\alpha_2}(w(\alpha_2)) + \mathcal{R}_{\alpha_3}(w(\alpha_2)) - w(\alpha_2) \\
\nonumber =&\ \frac{1}{k} \mathcal{R}_{k \alpha_1}(w(\alpha_2)) + w(\alpha_2) - (1 - \theta) + \frac{\alpha_2 + \alpha_3}{2} - w(\alpha_2) \\
\nonumber =&\ \frac{1}{k} \mathcal{R}_{k \alpha_1}(w(\alpha_2)) + \frac{(2 \theta - 1) - \alpha_1}{2} > \varepsilon^7 > \varepsilon^{10}
\end{align}
by \eqref{eq:R_alpha_explicit} and our assumption.

Assume next that $\frac{5(1 - \theta)}{8} \leqslant \alpha_2 \leqslant 1 - \theta \leqslant \alpha_3$. By above discussions, we know that $f(t) > \varepsilon^{10}$ for $t \in [\frac{\varepsilon^6}{2},\ \alpha_2 - \frac{1 - \theta}{2}]$, and $F_3(t)$ is non-increasing in $[\alpha_2 - \frac{1 - \theta}{2},\ w(\alpha_2)]$ and non-decreasing in $[w(\alpha_2),\ 1 - \theta]$. Thus, we only need to check the value of $F_3(w(\alpha_2))$. Now,
\begin{align}
\nonumber F_3(w(\alpha_2)) =&\ \frac{1}{k} \mathcal{R}_{k \alpha_1}(w(\alpha_2)) + \mathcal{R}_{\alpha_2}(w(\alpha_2)) + \mathcal{R}_{\alpha_3}(w(\alpha_2)) - w(\alpha_2) \\
\nonumber =&\ \frac{1}{k} \mathcal{R}_{k \alpha_1}(w(\alpha_2)) + w(\alpha_2) + \frac{\alpha_2 - (1 - \theta)}{2} - w(\alpha_2) \\
\nonumber =&\ \frac{1}{k} \mathcal{R}_{k \alpha_1}(w(\alpha_2)) + \frac{\alpha_2 - (1 - \theta)}{2} > \varepsilon^7 > \varepsilon^{10}
\end{align}
by \eqref{eq:R_alpha_explicit} and our assumption.

If $\alpha_2, \alpha_3 \geqslant 1 - \theta$, then we have $\mathcal{R}_{\alpha_2}(t) = \mathcal{R}_{\alpha_3}(t) = \frac{t}{2}$, and by \eqref{eq:f} we know that
$$
f(t) = P_1(t) \geqslant \frac{1}{\lceil \frac{1 - \theta}{\varepsilon} \rceil} \mathcal{R}_{\lceil \frac{1 - \theta}{\varepsilon} \rceil \alpha_1}(t) = \frac{t}{2 \lceil \frac{1 - \theta}{\varepsilon} \rceil} > \varepsilon^8 > \varepsilon^{10}
$$
since $\alpha_1 \geqslant \varepsilon$. Now the proof of Lemma~\ref{l34} is completed.
\end{proof}

\begin{lemma}\label{l35}
Suppose that $\alpha_1 + \alpha_2 + \alpha_3 = 1$ and $\alpha_1, \alpha_2, \alpha_3 \geqslant \varepsilon$. Then we have $f(t) > \varepsilon^{10}$ for every $t \in [\frac{\varepsilon^6}{2},\ 1 - \theta]$ if
$$
|\alpha_3 - \alpha_2| < 2 \theta - 1, \quad \alpha_3 \geqslant 1 - \theta \quad \text{and} \quad \alpha_1 \leqslant \frac{3(1 - \theta)}{4} - \varepsilon^5.
$$
\end{lemma}
\begin{proof}
If $\theta > 1 - \varepsilon$, then we have $1 - \theta < \varepsilon \leqslant \alpha_2, \alpha_3$, and we have $f(t) > \varepsilon^{10}$ by Lemma~\ref{l34}. Now we can assume that $\theta \leqslant 1 - \varepsilon$. Assume that $\alpha_3 \geqslant 1 - \theta$. If $\alpha_2 \geqslant 1 - \theta$, then we have $f(t) > \varepsilon^{10}$ by Lemma~\ref{l34}. Now we can assume that $\alpha_2 < 1 - \theta$. Since $|\alpha_3 - \alpha_2| < 2 \theta - 1$ and $\alpha_1 \leqslant \frac{3(1 - \theta)}{4}$, we have
$$
\alpha_2 = \frac{1 - \alpha_1 - |\alpha_3 - \alpha_2|}{2} > \frac{1 - \alpha_1 - 2 \theta + 1}{2} = 1 - \theta - \frac{\alpha_1}{2} \geqslant \frac{5(1 - \theta)}{8}.
$$
By the discussions in Lemma~\ref{l34}, we only need to check the value of $F_1(w(\alpha_2))$. Now,
\begin{align}
\nonumber F_3(w(\alpha_2)) =&\ \frac{1}{k} \mathcal{R}_{k \alpha_1}(w(\alpha_2)) + \mathcal{R}_{\alpha_2}(w(\alpha_2)) + \mathcal{R}_{\alpha_3}(w(\alpha_2)) - w(\alpha_2) \\
\nonumber =&\ \frac{1}{k} \mathcal{R}_{k \alpha_1}(w(\alpha_2)) + w(\alpha_2) + \frac{\alpha_2 - (1 - \theta)}{2} - w(\alpha_2) \\
\nonumber =&\ \frac{1}{k} \mathcal{R}_{k \alpha_1}(w(\alpha_2)) + \frac{\alpha_2 - (1 - \theta)}{2} = G(\alpha_2).
\end{align}
Note that $w(\alpha)$ has a slope $-\frac{2}{5}$ or $-1$, for fixed $k, \theta, \alpha_1$ and any $1 - \theta - \frac{\alpha_1}{2} \leqslant z_1 < z_2 \leqslant 1 - \theta$ we have
$$
G(z_2) - G(z_1) \geqslant \frac{z_2 - z_1}{2} - \frac{w(z_1) - w(z_2)}{2 k} \geqslant \frac{k - 1}{2 k} (z_2 - z_1) > 0.
$$
Now we only need to find an integer $k \geqslant 2$ such that $G\left(1 - \theta - \frac{\alpha_1}{2} \right) > \varepsilon^7$. We consider the following 3 cases:

(1). $\varepsilon \leqslant \alpha_1 < \frac{1 - \theta}{3}$. Now $1 - \theta - \frac{\alpha_1}{2} > \frac{5(1 - \theta)}{6}$, and
$$
w\left(1 - \theta - \frac{\alpha_1}{2} \right) = \frac{3(1 - \theta)}{2} - \left(1 - \theta - \frac{\alpha_1}{2} \right) = \frac{1 - \theta + \alpha_1}{2}.
$$
Write $l = \lfloor\frac{1 - \theta}{\alpha_1} \rfloor$. Taking $k = l \geqslant 3$. Then
\begin{align}
\nonumber G\left(1 - \theta - \frac{\alpha_1}{2} \right) \geqslant&\ - \frac{\alpha_1}{4} + \frac{w\left(1 - \theta - \frac{\alpha_1}{2} \right) - (1 - \theta) + l \alpha_1}{2 l} \\
\nonumber =&\ - \frac{\alpha_1}{4} + \frac{1 - \theta + \alpha_1 - 2 (1 - \theta - l \alpha_1)}{4 l} \\
\nonumber \geqslant&\ \frac{- l \alpha_1 + 1 - \theta + \alpha_1 - 2 (1 - \theta - l \alpha_1)}{4 l} \\
\nonumber =&\ \frac{\alpha_1 - (1 - \theta - l \alpha_1)}{4 l}.
\end{align}
Taking $k = l + 1 \geqslant 4$. Since $l + 1 \geqslant \lceil\frac{1 - \theta}{\alpha_1} \rceil$, we have $\mathcal{R}_{k \alpha_1}\left(w\left(1 - \theta - \frac{\alpha_1}{2} \right)\right) = \frac{w\left(1 - \theta - \frac{\alpha_1}{2} \right)}{2}$. Then
\begin{align}
\nonumber G\left(1 - \theta - \frac{\alpha_1}{2} \right) =&\ - \frac{\alpha_1}{4} + \frac{1 - \theta + \alpha_1}{4 (l + 1)} \\
\nonumber =&\ \frac{1 - \theta - l \alpha_1}{4 (l + 1)}.
\end{align}
Combining the two lower bounds, we have
\begin{align}
\nonumber G\left(1 - \theta - \frac{\alpha_1}{2} \right) \geqslant&\ \frac{1}{2 l + 1}\left(l \frac{\alpha_1 - (1 - \theta - l \alpha_1)}{4 l} + (l + 1) \frac{1 - \theta - l \alpha_1}{4 (l + 1)} \right) \\
\nonumber =&\ \frac{\alpha_1}{4(2 l + 1)}.
\end{align}
Since $l \leqslant \frac{1 - \theta}{\alpha_1}$, we have
$$
G\left(1 - \theta - \frac{\alpha_1}{2} \right) \geqslant \frac{\alpha_1}{4(2 l + 1)} \geqslant \frac{\alpha_1^2}{4(2 (1 - \theta) + \alpha_1)} > \frac{\alpha_1^2}{12 (1 - \theta)} > \varepsilon^3 > \varepsilon^7.
$$

(2). $\frac{1 - \theta}{3} \leqslant \alpha_1 < \frac{1 - \theta}{2}$. Now $\frac{3(1 - \theta)}{4} < 1 - \theta - \frac{\alpha_1}{2} \leqslant \frac{5(1 - \theta)}{6}$, and
$$
w\left(1 - \theta - \frac{\alpha_1}{2} \right) = 1 - \theta - \frac{2}{5}\left(1 - \theta - \frac{\alpha_1}{2} \right) = \frac{3(1 - \theta) + \alpha_1}{5}.
$$
Since
$$
G\left(1 - \theta - \frac{\alpha_1}{2} \right) \geqslant - \frac{\alpha_1}{4} + \frac{1}{2} \mathcal{R}_{2 \alpha_1}\left(w\left(1 - \theta - \frac{\alpha_1}{2} \right)\right),
$$
we only need to prove that
$$
\mathcal{R}_{2 \alpha_1}\left(w\left(1 - \theta - \frac{\alpha_1}{2} \right)\right) > \frac{\alpha_1}{2} + \varepsilon^6.
$$
By \eqref{eq:R_alpha}, this follows from the 3 conditions:
$$
\frac{w\left(1 - \theta - \frac{\alpha_1}{2} \right)}{2} - \frac{\alpha_1}{2} = \frac{3 (1 - \theta) - 4 \alpha_1}{10} > \varepsilon^6,
$$
$$
\frac{w\left(1 - \theta - \frac{\alpha_1}{2} \right) - (1 - \theta) + \frac{16}{5} \alpha_1}{4} - \frac{\alpha_1}{2} = \frac{7 \alpha_1 - 2 (1 - \theta)}{20} > \varepsilon^6,
$$
$$
\frac{w\left(1 - \theta - \frac{\alpha_1}{2} \right) - \frac{1 - \theta}{2} + 2 \alpha_1}{4} - \frac{\alpha_1}{2} = \frac{(1 - \theta) + 2 \alpha_1}{40} > \varepsilon^6,
$$
since $\frac{2 (1 - \theta)}{7} + \varepsilon^5 < \frac{1 - \theta}{3} \leqslant \alpha_1 < \frac{1 - \theta}{2} < \frac{3(1 - \theta)}{4} - \varepsilon^5$.

(3). $\frac{1 - \theta}{2} \leqslant \alpha_1 \leqslant \frac{3 (1 - \theta)}{4} - \varepsilon^5$. Now $\frac{5(1 - \theta)}{8} < 1 - \theta - \frac{\alpha_1}{2} \leqslant \frac{3(1 - \theta)}{4}$, and
$$
w\left(1 - \theta - \frac{\alpha_1}{2} \right) = 1 - \theta - \frac{2}{5}\left(1 - \theta - \frac{\alpha_1}{2} \right) = \frac{3(1 - \theta) + \alpha_1}{5}.
$$
Now we have
$$
\mathcal{R}_{2 \alpha_1}\left(w\left(1 - \theta - \frac{\alpha_1}{2} \right)\right) - \frac{\alpha_1}{2} = \frac{w\left(1 - \theta - \frac{\alpha_1}{2} \right)}{2} - \frac{\alpha_1}{2} = \frac{3 (1 - \theta) - 4 \alpha_1}{10} > \varepsilon^6.
$$

Combining the 3 cases, the proof of Lemma~\ref{l35} is completed.
\end{proof}

\begin{lemma}\label{l36}
Suppose that $\alpha_1 + \alpha_2 + \alpha_3 = 1$ and $\alpha_1, \alpha_2, \alpha_3 \geqslant \varepsilon$. Then we have $f(t) > \varepsilon^{10}$ for every $t \in [\frac{\varepsilon^6}{2},\ 1 - \theta]$ if
$$
\alpha_1 \leqslant \alpha_2 \leqslant 1 - \theta \leqslant \alpha_3 \quad \text{and} \quad \alpha_1 + \alpha_2 \geqslant \frac{5(1 - \theta)}{4} + \varepsilon^5.
$$
\end{lemma}
\begin{proof}
If $\theta > 1 - \varepsilon$, then we have $1 - \theta < \varepsilon \leqslant \alpha_2, \alpha_3$, and we have $f(t) > \varepsilon^{10}$ by Lemma~\ref{l34}. Now we can assume that $\theta \leqslant 1 - \varepsilon$. Since $\alpha_1 \leqslant \alpha_2$ and $\alpha_1 + \alpha_2 \geqslant \frac{5(1 - \theta)}{4} + \varepsilon^5$, we have $\alpha_2 \geqslant \frac{5(1 - \theta)}{8} + \frac{\varepsilon^5}{2}$.

If $\alpha_1 \geqslant \frac{1 - \theta}{2}$, then $\frac{1}{2} \mathcal{R}_{2 \alpha_1}(w(\alpha_2)) = \frac{w(\alpha_2)}{4}$, and $\frac{w(\alpha_2)}{4} - \frac{1 - \theta - \alpha_2}{2} = \phi(\alpha_2) \geqslant \min(\frac{\varepsilon^5}{5}, \frac{1 - \theta}{12}) > \varepsilon^7$ by Lemma~\ref{l22}. Now we have $f(t) > \varepsilon^{10}$ by Lemma~\ref{l34}.

Let $\alpha_1 < \frac{1 - \theta}{2}$. Since $\alpha_1 + \alpha_2 \geqslant \frac{5(1 - \theta)}{4} + \varepsilon^5$, we have $\frac{1 - \theta}{2} \leqslant \frac{5(1 - \theta)}{2} - 2 \alpha_2 < 2 \alpha_1$. Since $\mathcal{R}_{\alpha}$ is non-decreasing, we have
$$
\mathcal{R}_{2 \alpha_1}(w(\alpha_2)) \geqslant \mathcal{R}_{\frac{5(1 - \theta)}{2} - 2 \alpha_2}(w(\alpha_2)).
$$
Now we only need to show that
$$
\mathcal{R}_{\frac{5(1 - \theta)}{2} - 2 \alpha_2}(w(\alpha_2)) \geqslant 1 - \theta - \alpha_2 + \varepsilon^7.
$$
If $\alpha_2 \leqslant \frac{5(1 - \theta)}{6}$, then $w(\alpha_2) = 1 - \theta - \frac{2}{5} \alpha_2$ and
$$
\frac{w(\alpha_2)}{2} - (1 - \theta - \alpha_2) = \frac{4}{5} \alpha_2 - \frac{1 - \theta}{2} \geqslant \frac{2}{5} \varepsilon^5 > \varepsilon^7,
$$
$$
\frac{w(\alpha_2) - (1 - \theta) + \frac{5(1 - \theta)}{2} - 2 \alpha_2}{2} - (1 - \theta - \alpha_2) = \frac{1 - \theta}{4} - \frac{1}{5} \alpha_2 \geqslant \frac{1 - \theta}{12} \geqslant \frac{\varepsilon}{12} > \varepsilon^7.
$$
If $\alpha_2 > \frac{5(1 - \theta)}{6}$, then $w(\alpha_2) = \frac{3(1 - \theta)}{2} - \alpha_2$ and
$$
\frac{w(\alpha_2)}{2} - (1 - \theta - \alpha_2) = \frac{1}{2} \alpha_2 - \frac{1 - \theta}{4} \geqslant \frac{1 - \theta}{16} \geqslant \frac{\varepsilon}{16} > \varepsilon^7,
$$
$$
\frac{w(\alpha_2) - (1 - \theta) + \frac{8}{5} \left(\frac{5(1 - \theta)}{2} - 2 \alpha_2\right)}{4} - (1 - \theta - \alpha_2) = \frac{1 - \theta}{8} - \frac{1}{20} \alpha_2 \geqslant \frac{3(1 - \theta)}{40} \geqslant \frac{3 \varepsilon}{40} > \varepsilon^7,
$$
$$
\frac{w(\alpha_2) - \frac{1 - \theta}{2} + \frac{5(1 - \theta)}{2} - 2 \alpha_2}{4} - (1 - \theta - \alpha_2) = \frac{1}{4} \alpha_2 - \frac{1 - \theta}{8} \geqslant \frac{1 - \theta}{32} \geqslant \frac{\varepsilon}{32} > \varepsilon^7.
$$
In both cases we have $\mathcal{R}_{\frac{5(1 - \theta)}{2} - 2 \alpha_2}(w(\alpha_2)) \geqslant 1 - \theta - \alpha_2 + \varepsilon^7$, hence $f(t) > \varepsilon^{10}$ by Lemma~\ref{l34}. Now the proof of Lemma~\ref{l36} is completed.
\end{proof}

\section{Proof of Theorem 1 and Theorem 2}
By Lemma~\ref{MPPTW_Proposition3.2}, we know that Theorem~\ref{t2} is a direct corollary of Lemma~\ref{l34}. In order to prove Theorem~\ref{t1}, we first suppose that $\alpha_1 < \gamma^{*}(\theta) - \varepsilon$. Without loss of generality, we can assume that $\alpha_2 \leqslant \alpha_3$, and thus $\alpha_3 - \alpha_2 < 2 \theta - 1$. We split the range of $\theta$ into two subranges.

\subsection{The range $\frac{1}{2} < \theta \leqslant \frac{31}{55}$}
Now we have $\gamma^{*}(\theta) = \max\left(\Gamma_{0}(\theta),\ u_{k_0(\theta)}(\theta) \right)$ and $\alpha_1 < \max\left(\Gamma_{0}(\theta),\ u_{k_0(\theta)}(\theta) \right) - \varepsilon$. Since $u_k(\theta) = 2 \theta - 1 + \frac{1 - \theta}{2 k - 1} \leqslant u_2(\theta)$, by Lemma~\ref{l24} we have $\alpha_1 < u_2(\theta) = \frac{5 \theta - 2}{3} < \frac{3(1 - \theta)}{4} - \varepsilon^5$ ($\theta < \frac{17}{29}$). If $\alpha_3 \geqslant 1 - \theta$, the desired result holds by Lemma~\ref{l35}. Suppose that $\alpha_3 < 1 - \theta$. Since $\alpha_2 \leqslant \alpha_3$, we have $\alpha_1 > 1 - 2 \alpha_3 = 2 \theta - 1$, and $\frac{\alpha_1 - (2 \theta - 1)}{2} > 0$. Since $\alpha_2 = \frac{1 - \alpha_1 - (\alpha_3 - \alpha_2)}{2} > \frac{1 - \alpha_1 - (2 \theta - 1)}{2} = 1 - \theta - \frac{\alpha_1}{2} > \frac{5(1 - \theta)}{8}$, we have $\frac{5(1 - \theta)}{8} < \alpha_2 \leqslant \alpha_3 < 1 - \theta$. By Lemma~\ref{l34}, we need to find $k \geqslant 2$ such that
$$
\frac{1}{k} \mathcal{R}_{k \alpha_1}(w(\alpha_2)) > \frac{\alpha_1 - (2 \theta - 1)}{2} + \varepsilon^7.
$$
Since $\alpha_2 \leqslant \frac{1 - \alpha_1}{2}$ and $\frac{1 - \alpha_1}{2} > \frac{5(1 - \theta)}{6}$ ($\alpha_1 > \frac{5 \theta - 2}{3}$), note that $w$ is decreasing and $\mathcal{R}_{k \alpha_1}$ is non-decreasing, we only need to show that
$$
\frac{1}{k} \mathcal{R}_{k \alpha_1}\left(w\left(\frac{1 - \alpha_1}{2}\right)\right) = \frac{1}{k} \mathcal{R}_{k \alpha_1}\left(\frac{2 - 3 \theta + \alpha_1}{2}\right) > \frac{\alpha_1 - (2 \theta - 1)}{2} + \varepsilon^7.
$$
By \eqref{eq:R_alpha}, we have
\begin{align}
\nonumber &\ \frac{1}{k} \mathcal{R}_{k \alpha_1}\left(\frac{2 - 3 \theta + \alpha_1}{2}\right) - \frac{\alpha_1 - (2 \theta - 1)}{2} \\
\nonumber \geqslant&\ \frac{1}{k} \min\left(\frac{2 - 3 \theta + \alpha_1}{4} - k \frac{\alpha_1 - (2 \theta - 1)}{2}, \right. \\
\nonumber &\ \qquad \qquad \qquad \max\left(\frac{\frac{2 - 3 \theta + \alpha_1}{2} - (1 - \theta) + k \alpha_1}{2} - k \frac{\alpha_1 - (2 \theta - 1)}{2}, \right. \\
\nonumber &\ \qquad \qquad \qquad \qquad \min\left(\frac{\frac{2 - 3 \theta + \alpha_1}{2} - \frac{1 - \theta}{2} + k \alpha_1}{4} - k \frac{\alpha_1 - (2 \theta - 1)}{2}, \right. \\
&\ \left. \left. \left. \qquad \qquad \qquad \qquad \qquad \qquad \frac{\frac{2 - 3 \theta + \alpha_1}{2} - (1 - \theta) + \frac{8}{5} k \alpha_1}{4} - k \frac{\alpha_1 - (2 \theta - 1)}{2} \right) \right) \right). \label{eq:14}
\end{align}
When $\alpha_1 < \Gamma_{0}(\theta) - \varepsilon$, we choose $k \geqslant 2$ such that $\min\left(u_k(\theta), p_k(\theta), q_k(\theta) \right)$ reach the maximum value. By Lemma~\ref{l23} and \eqref{eq:14}, we have
\begin{align}
\nonumber &\ \frac{1}{k} \mathcal{R}_{k \alpha_1}\left(\frac{2 - 3 \theta + \alpha_1}{2}\right) - \frac{\alpha_1 - (2 \theta - 1)}{2} \\
\nonumber \geqslant&\ \frac{1}{k} \min\left(\frac{(2 k - 1)(u_k(\theta) - \alpha_1)}{4},\ \frac{(4 k - 5)(p_k(\theta) - \alpha_1)}{40},\ \frac{(2 k - 1)(q_k(\theta) - \alpha_1)}{8} \right) \geqslant \frac{3 \varepsilon}{80} > \varepsilon^7,
\end{align}
and the desired result holds. When $\Gamma_{0}(\theta) - \varepsilon \leqslant \alpha_1 < u_{k_0(\theta)}(\theta)$ (this only occurs when $u_{k_0(\theta)}(\theta) > \Gamma_{0}(\theta)$), we let $k = k_0(\theta)$. By the definition of $k_0(\theta)$, one can easily see that $k_0(\theta)$ is the smallest integer $k \geqslant 2$ such that $l_k(\theta) < \Gamma_{0}(\theta)$. Since $k_0(\theta)$, $\Gamma_{0}(\theta)$ and $l_k(\theta)$ are fixed for fixed $\theta$, we can choose $\varepsilon > 0$ such that $l_k(\theta) + 2 \varepsilon < \Gamma_{0}(\theta)$. Now we have $l_k(\theta) < \Gamma_{0}(\theta)$, and thus $\alpha_1 - l_k(\theta) \geqslant \Gamma_{0}(\theta) - l_k(\theta) - \varepsilon > \varepsilon$. Now by Lemma~\ref{l23} and \eqref{eq:14}, we have
\begin{align}
\nonumber &\ \frac{1}{k} \mathcal{R}_{k \alpha_1}\left(\frac{2 - 3 \theta + \alpha_1}{2}\right) - \frac{\alpha_1 - (2 \theta - 1)}{2} \\
\nonumber \geqslant&\ \frac{1}{k} \min\left(\frac{(2 k - 1)(u_k(\theta) - \alpha_1)}{4},\ \frac{\alpha_1 - l_k(\theta)}{4} \right) \geqslant \frac{\varepsilon}{4k} > \varepsilon^7,
\end{align}
and the desired result holds. Now the proof of Theorem~\ref{t1} in the range $\frac{1}{2} < \theta \leqslant \frac{31}{55}$ is completed.

\subsection{The range $\frac{31}{55} < \theta < 1$}
Now we have $\alpha_1 < \gamma^{*}(\theta) - \varepsilon \leqslant \frac{5 \theta - 1}{4} - \varepsilon$, and $\alpha_3 \geqslant \frac{1 - \alpha_1}{2} > \frac{5(1 - \theta)}{8} + \frac{\varepsilon}{2}$. If $\theta \geqslant 1 - \varepsilon$, then we have $\alpha_2, \alpha_3 \geqslant \varepsilon \geqslant 1 - \theta$, and by Lemma~\ref{l34} the desired result holds. Hence we can assume that $\theta < 1 - \varepsilon$. By \eqref{eq:phi5/8}, we know that $\phi(\alpha_3) \geqslant \frac{\varepsilon}{12}$. If $\theta \leqslant \frac{17}{30}$, we also have $\alpha_3 \geqslant \frac{1 - \alpha_1}{2} > 6 - 10 \theta + \frac{\varepsilon}{2}$.

If $\alpha_1 \geqslant 1 - \theta$, we can assume that $\alpha_2 \leqslant \alpha_3 < 1 - \theta$. Since we have $\alpha_2 + \alpha_3 = 1 - \alpha_1 > \frac{5(1 - \theta)}{4} + \varepsilon > \frac{5(1 - \theta)}{4} + \varepsilon^5$, by Lemma~\ref{l36} the desired result holds.

If $\alpha_1 < 1 - \theta \leqslant \alpha_3$, we can assume that $\alpha_2 < 1 - \theta$. If $\alpha_1 \leqslant \frac{3(1 - \theta)}{4} - \varepsilon^5$, then by Lemma~\ref{l35} the desired result holds. If $\alpha_1 > \frac{3(1 - \theta)}{4} - \varepsilon^5$, we have
$$
\alpha_1 + \alpha_2 > \alpha_1 + \left(1 - \theta - \frac{\alpha_1}{2}\right) = \frac{11(1 - \theta)}{8} - \frac{\varepsilon^5}{2},
$$
and by Lemma~\ref{l36} the desired result holds.

If $\alpha_1, \alpha_3 < 1 - \theta$, we have $\alpha_2 < \alpha_3 < 1 - \theta$ and $\theta < \frac{2}{3}$. Now we have $\alpha_1 > 1 - 2 \alpha_3 > 2 \theta - 1$ and $\alpha_2 > 1 - \theta - \frac{\alpha_1}{2} > \frac{1 - \theta}{2}$. If $\alpha_2 < \frac{5(1 - \theta)}{8} + \varepsilon$, we have $\alpha_1 > 1 - \alpha_2 - \alpha_3 > 1 - \frac{13(1 - \theta)}{8} - \varepsilon > \frac{5(1 - \theta)}{8} + \varepsilon$ and $\phi(\alpha_1) \geqslant \frac{\varepsilon}{12}$, since $\theta > \frac{5}{9}$. Now Lemma~\ref{l33} gives
$$
f(t) \geqslant \min\left(\frac{\varepsilon^6}{8},\ \phi(\alpha_1) + \phi(\alpha_3) - \frac{w(\max(\alpha_1, \alpha_3))}{4} \right).
$$
If $\alpha_1, \alpha_3 \leqslant \frac{5(1 - \theta)}{6}$, then
\begin{align}
\nonumber \phi(\alpha_1) + \phi(\alpha_3) - \frac{w(\max(\alpha_1, \alpha_3))}{4} =&\ \frac{2}{5}(1 - \alpha_2) - \frac{3(1 - \theta)}{4} + \frac{\max(\alpha_1, \alpha_3)}{10} \\
\nonumber \geqslant&\ \frac{9}{20}(1 - \alpha_2) - \frac{3(1 - \theta)}{4} \\
\nonumber >&\ \frac{9}{20} - \frac{33(1 - \theta)}{32} - \varepsilon = \frac{33}{32} \left(\theta - \frac{31}{55}\right) - \varepsilon.
\end{align}
If exactly one of $\alpha_1$ and $\alpha_3$ is $> \frac{5(1 - \theta)}{6}$, then
\begin{align}
\nonumber \phi(\alpha_1) + \phi(\alpha_3) - \frac{w(\max(\alpha_1, \alpha_3))}{4} =&\ \left( \frac{\max(\alpha_1, \alpha_3)}{4} - \frac{1 - \theta}{8} \right) + \left( \frac{2 \min(\alpha_1, \alpha_3)}{5} - \frac{1 - \theta}{4} \right) \\
\nonumber & - \frac{1}{4}\left(\frac{3(1 - \theta)}{2} - \max(\alpha_1, \alpha_3) \right) \\
\nonumber =&\ \frac{2}{5}(1 - \alpha_2) - \frac{3(1 - \theta)}{4} + \frac{\max(\alpha_1, \alpha_3)}{10} \\
\nonumber >&\ \frac{2}{5} - \frac{1 - \theta}{4} - \frac{3(1 - \theta)}{4} + \frac{1 - \theta}{12} - \varepsilon = \frac{11}{12}\left(\theta - \frac{31}{55}\right) - \varepsilon.
\end{align}
If $\alpha_1, \alpha_3 > \frac{5(1 - \theta)}{6}$, then
\begin{align}
\nonumber \phi(\alpha_1) + \phi(\alpha_3) - \frac{w(\max(\alpha_1, \alpha_3))}{4} =&\ \frac{\alpha_1 + \alpha_3 + \max(\alpha_1, \alpha_3)}{4} - \frac{5(1 - \theta)}{8} \\
\nonumber \geqslant&\ \frac{3}{8}(1 - \alpha_2) - \frac{5(1 - \theta)}{8} \\
\nonumber >&\ \frac{3}{8} - \frac{55(1 - \theta)}{64} - \varepsilon = \frac{55}{64}\left(\theta - \frac{31}{55}\right) - \varepsilon.
\end{align}
Since for every $\theta > \frac{31}{55}$ we can choose $\varepsilon > 0$ such that $\theta - \frac{31}{55} > 10 \varepsilon$, the desired result follows from Lemma~\ref{l33}.

If $\alpha_1, \alpha_2 \geqslant \frac{5(1 - \theta)}{8} + \varepsilon$, we have $\frac{5(1 - \theta)}{8} + \varepsilon \leqslant \alpha_1, \alpha_2, \alpha_3 < 1 - \theta$ and $\phi(\alpha_j) \geqslant \frac{\varepsilon}{12}$ for $j = 1, 2, 3$. Now Lemma~\ref{l31} gives
$$
f(t) \geqslant \min\left(\frac{\varepsilon^6}{4},\ \phi(\alpha_1) + \phi(\alpha_2) + \phi(\alpha_3) - \frac{w(\max(\alpha_1, \alpha_3))}{4} \right).
$$
If exactly two of the three $\alpha_j$ are $> \frac{5(1 - \theta)}{6}$, by Lemma~\ref{l32} we have
$$
\phi(\alpha_1) + \phi(\alpha_2) + \phi(\alpha_3) - \frac{w(\max(\alpha_1, \alpha_3))}{4} \geqslant \frac{3}{8} - \frac{55}{64}(1 - \theta) = \frac{55}{64}\left(\theta - \frac{31}{55}\right) > \varepsilon,
$$
since for every $\theta > \frac{31}{55}$ we can choose $\varepsilon > 0$ such that $\theta - \frac{31}{55} > 10 \varepsilon$. If all three $\alpha_j$ are $> \frac{5(1 - \theta)}{6}$, by Lemma~\ref{l32} we have
$$
\phi(\alpha_1) + \phi(\alpha_2) + \phi(\alpha_3) - \frac{w(\max(\alpha_1, \alpha_3))}{4} = \frac{1 - 3(1-\theta) + \alpha_3}{4} \geqslant \frac{1}{3} - \frac{3(1-\theta)}{4} > \varepsilon,
$$
since $\theta > \frac{5}{9} + 10 \varepsilon$. Otherwise we have, by Lemma~\ref{l32},
$$
\phi(\alpha_1) + \phi(\alpha_2) + \phi(\alpha_3) - \frac{w(\max(\alpha_1, \alpha_3))}{4} = \frac{\max(\alpha_1, \alpha_3) - (6 - 10 \theta)}{10}.
$$
Since $\max(\alpha_1, \alpha_3) \geqslant \frac{1}{3}$, if $\theta > \frac{17}{30}$, we can choose $\varepsilon > 0$ such that $\theta - \frac{17}{30} > \varepsilon$, and we have
$$
\frac{\max(\alpha_1, \alpha_3) - (6 - 10 \theta)}{10} > \varepsilon.
$$
If $\theta \leqslant \frac{17}{30}$, we know that $\max(\alpha_1, \alpha_3) \geqslant \alpha_3 \geqslant \frac{1 - \alpha_1}{2} > 6 - 10 \theta + \frac{\varepsilon}{2}$, and
$$
\frac{\max(\alpha_1, \alpha_3) - (6 - 10 \theta)}{10} \geqslant \frac{\varepsilon}{20}.
$$
Hence the desired result follows in all cases.

If $\alpha_1 < \frac{5(1 - \theta)}{8} + \varepsilon$ and $\alpha_2 \geqslant \frac{5(1 - \theta)}{8} + \varepsilon$, we have $2 \theta - 1 < \alpha_1 < \frac{5(1 - \theta)}{8} + \varepsilon$. Since for every $\theta > \frac{31}{55}$ we can choose $\varepsilon > 0$ such that $\theta - \frac{31}{55} > 10 \varepsilon$, we have
$$
u_2(\theta) - \alpha_1 = \frac{5 \theta - 2}{3} - \alpha_1 > \frac{5 \theta - 2}{3} - \frac{5(1 - \theta)}{8} - \varepsilon = \frac{55 \theta - 31}{24} - \varepsilon > \varepsilon
$$
and
$$
\alpha_1 - l_2(\theta) = \alpha_1 + 7 \theta - 4 > 9 \theta - 5 > \varepsilon.
$$
Since $\frac{5(1 - \theta)}{8} + \varepsilon \leqslant \alpha_2 \leqslant \alpha_3 < 1 - \theta$, by Lemma~\ref{l34}, we only need to show that
$$
\frac{1}{2} \mathcal{R}_{2 \alpha_1}(w(\alpha_2)) > \frac{\alpha_1 - (2 \theta - 1)}{2} + \varepsilon^7.
$$
Since $\alpha_2 \leqslant \frac{1 - \alpha_1}{2}$, $w$ is decreasing and $\mathcal{R}_{2 \alpha_1}$ is non-decreasing, we have
$$
\frac{1}{2} \mathcal{R}_{2 \alpha_1}(w(\alpha_2)) \geqslant \frac{1}{2} \mathcal{R}_{2 \alpha_1}\left(w\left(\frac{1 - \alpha_1}{2}\right)\right) = \frac{1}{2} \mathcal{R}_{2 \alpha_1}\left(\frac{2 - 3 \theta + \alpha_1}{2}\right).
$$
By \eqref{eq:R_alpha} and Lemma~\ref{l23}, we have
$$
\frac{1}{2} \mathcal{R}_{2 \alpha_1}\left(\frac{2 - 3 \theta + \alpha_1}{2}\right) - \frac{\alpha_1 - (2 \theta - 1)}{2} \geqslant \frac{1}{2} \min\left(\frac{3(u_2(\theta) - \alpha_1)}{4},\ \frac{\alpha_1 - l_2(\theta)}{4} \right) > \frac{\varepsilon}{8},
$$
and the desired result holds. Now the proof of Theorem~\ref{t1} in the range $\frac{31}{55} < \theta < 1$ is completed. Combining the two ranges, the proof of Theorem~\ref{t1} is completed.

\section*{AI declaration}
ChatGPT-6 Astra (OpenAI) and Claude Opus 5.5 (Anthropic) were used in some numerical calculations and verifications in the proof of Theorem~\ref{t1} and Theorem~\ref{t2}, and in the preparation of the manuscript.

\bibliographystyle{plain}
\bibliography{bib}
\end{document}